\documentclass[12pt]{amsart}
\usepackage{amssymb}
\usepackage{amsfonts}
\usepackage{amsthm}
\usepackage{graphicx}
\usepackage{color}
\usepackage{url}
\usepackage{microtype}
\usepackage{tikz}

\usepackage{hyperref}

\usepackage{pgfplots}
\pgfplotsset{compat=1.5}
\pgfplotsset{width=10cm}
\pgfplotsset{height=6cm}

\definecolor{darkblue}{RGB}{0,0,160}
\hypersetup{
colorlinks,%
citecolor=black,%
filecolor=black,%
linkcolor=darkblue,%
urlcolor=darkblue
}

\usepackage[utf8]{inputenc}
\usepackage{todonotes}

\newtheorem{thm}{Theorem}[section]
\newtheorem{lemma}[thm]{Lemma}

\newtheorem{cor}[thm]{Corollary}
\newtheorem{prop}[thm]{Proposition}

\theoremstyle{definition}
\newtheorem{example}[thm]{Example}
\newtheorem{remark}[thm]{Remark}
\newtheorem{defn}[thm]{Definition}

\numberwithin{equation}{section}

\newcommand{\ring}[1]{\ensuremath{\mathbb{#1}}}

\newcommand\NN{\ring{N}}

\newcommand\QQ{\ring{Q}}
\newcommand\RR{\ring{R}}

\newcommand\ZZ{\ring{Z}}

\newcommand\conv{\mathrm{conv}}

\newcommand\cB{{\mathcal B}}
\newcommand\cC{{\mathcal C}}

\newcommand\cF{{\mathcal F}}
\newcommand\cI{{\mathcal I}}

\newcommand\cM{{\mathcal M}}

\newcommand\cG{{\mathcal G}}
\newcommand\cP{{\mathcal P}}

\newcommand\cone[1]{\NN#1}

\DeclareMathOperator\rank{rank} 
\DeclareMathOperator{\supp}{supp} 

\newcommand{\inD}[1][\relax]{\def\argone{#1}\def\temprelax{\relax}
  \ifx\argone\temprelax\right.\else\,\middle|#1\right.{}\fi}

\newcommand{\diam}[1]{\mathrm{diam}\left( #1 \right)}
\newcommand{\fdiam}[3]{d^{#1}_{#2,#3}}
\newcommand{\graver}[1]{\textnormal{Gr}_{#1}}
\newcommand{\relaxfiber}[2]{\mathcal{R}_{#1,#2}}

\newcommand{\boundaryDir}[3]{\partial_{#1}^{#2}(#3)}

\newcommand{\fiber}[2]{\mathcal{F}_{#1,#2}}
\newcommand{\fibergraph}[3]{\fiber{#1}{#2}(#3)}

\pgfplotscreateplotcyclelist{mycolorlist}{%
black,every mark/.append style={fill=black!80!black},mark=*\\%
black,every mark/.append style={fill=black!80!black},mark=square*\\%
}

\pgfplotstableread[row sep=crcr]{
1 0.75\\
2 0.853553\\
3 0.904508\\
4 0.933013\\
5 0.950484\\
6 0.96194\\
7 0.969846\\
8 0.975528\\
9 0.979746\\
10 0.982963\\
11 0.985471\\
12 0.987464\\
13 0.989074\\
14 0.990393\\
15 0.991487\\
16 0.992404\\
17 0.993181\\
18 0.993844\\
19 0.994415\\
20 0.994911\\
}\fPartition

\pgfplotstableread[row sep=crcr]{
1 0.75\\
2 0.85\\
3 0.880952\\
4 0.895833\\
5 0.904545\\
6 0.910256\\
7 0.914286\\
8 0.917279\\
9 0.919591\\
10 0.921429\\
11 0.922925\\
12 0.924167\\
13 0.925214\\
14 0.926108\\
15 0.926882\\
16 0.927557\\
17 0.928151\\
18 0.928679\\
19 0.92915\\
20 0.929573\\
}\fPartitionAdapted

\begin{document}

\title{Rapid mixing and Markov bases}

\author{Tobias Windisch}
\address{Otto-von-Guericke Universität\\ Magdeburg, Germany} 
\email{windisch@ovgu.de}

\date{\today}

\makeatletter
  \@namedef{subjclassname@2010}{\textup{2010} Mathematics Subject Classification}
\makeatother

\subjclass[2010]{Primary: 05C81, Secondary: 13P25}

\keywords{Random walks, fiber graphs, Markov bases, expander graphs}

\begin{abstract}
The mixing behaviour of random walks on lattice points of polytopes
using Markov bases is examined. It is shown that under a dilation of the
underlying polytope, these random walks do not mix rapidly when a
fixed Markov basis is used. We also show that this phenomenon does not
disappear after adding more moves to the Markov basis.  Avoiding
rejections by sampling applicable moves does also not lead to an
asymptotic improvement. As a way out, a method of how to adapt Markov
bases in order to achieve the fastest mixing behaviour is introduced.
\end{abstract}
\maketitle
\setcounter{tocdepth}{1}

\section{Introduction}
Random walks have been successfully used in various applications to
explore combinatorial structures where a complete enumeration is
computationally
prohibitive~\cite{Jerrum2004,Matthews1991,Diaconis1998a}. In many of
these applications, the underlying discrete objects correspond to the
elements of a \emph{fiber} $\fiber{A}{b}:=\{u\in\NN^d: Au=b\}$ of a
matrix $A\in\ZZ^{m\times d}$ with $\ker_\ZZ(A)\cap\NN^d=\{0\}$ and a
right-hand side $b\in\ZZ^m$.  The exploration of fibers with random
walks requires to connect their elements by edges so that there is a
path between any two of them. In their groundbreaking work~\cite{Diaconis1998a}, Diaconis
and Sturmfels have shown how to endow
$\fiber{A}{b}$ with the structure of a connected graph in a
computational way: For a finite set $\cM\subset\ker_\ZZ(A)$, the
\emph{fiber graph} $\fibergraph{A}{b}{\cM}$ is the graph on
$\fiber{A}{b}$ in which two nodes $u,v\in\fiber{A}{b}$ are adjacent if
$u-v\in\pm\cM$. They have coined the term \emph{Markov basis}, which
denotes a finite set $\cM\subset\ker_\ZZ(A)$ such that
$\fibergraph{A}{b}{\cM}$ is connected for all $b\in\ZZ^m$.  Their main
result shows that a Markov basis can be obtained by a Gröbner basis
computation in a polynomial ring~\cite[Theorem~3.1]{Diaconis1998a}.
Markov bases can be used to enumerate locally the neighborhood of a
node in the fiber graph, which makes them a general machinery to
approximate any probability distribution on any fiber of a matrix.
The number of steps needed to approximate a given distribution
sufficiently is the \emph{mixing time} of the random walk. 
Even though the computation of Markov bases received a lot of
attention in the last
decade~\cite{Hara2008,Kahle2014a,Develin2003,Sullivant2005,Sullivant2007},
mixing results on fiber graphs are still rare. It was shown
in~\cite{Diaconis1995} that the mixing time of random walks on two-way
contingency tables with the same row and column sums using a minimal
Markov basis is quadratic in the diameter of the underlying fiber
graph and a similar result is true for random walks on lattice points
of polytopes that use the unit vectors as Markov basis
vectors~\cite{Diaconis1996,Virag1998}.

In this paper, we study the mixing behaviour of the \emph{simple walk}
on fiber graphs, whose stationary distribution is the uniform
distribution on $\fiber{A}{b}$.  Our main result concerns the mixing
behaviour of fiber graph sequences that use a fixed Markov basis:

\begin{thm}\label{t:NotRapidlyMixing}
Let $A\in\ZZ^{m\times d}$, let $\cM\subset\ker_\ZZ(A)$ be a Markov
basis for $A$, and let $(b_i)_{i\in\NN}$ a dominated sequence in
$\cone{A}$. Then
$(\fibergraph{A}{b_i}{\cM})_{i\in\NN}$ is no
expander. If additionally $(b_i)_{i\in\NN}$ has a meaningful
parametrization, then $(\fibergraph{A}{b_i}{\cM})_{i\in\NN}$ is not
rapidly mixing.
\end{thm}

Surprisingly, walking randomly on a fiber with a larger Markov basis
(Remark~\ref{r:AddingMultiples}) or avoiding rejections by sampling
only applicable moves (Remark~\ref{r:OnlyApplicableMoves}) does not
improve the asymptotic mixing behaviour.
The conclusion we
draw from these results is that an adaption of the Markov basis has to
take place depending on the right-hand side $b\in\ZZ^m$. 
In Section~\ref{s:ConstructExpanders}, we adapt the Markov basis so
that the underlying graph becomes the complete graph with additional
loops. Adding more moves to a Markov basis
increases the rejection rate, i.e. the number of loops, on every node
of the fiber. Thus, it is a fine line to find the proper number of
moves to add without slowing down the random walk.  The fastest mixing
behaviour is obtained for expander graphs~\cite{Hoory2006} and we show
how to obtain expanders on fibers under mild assumptions on the
diameter (Corollary~\ref{c:LinearDiameter}). The idea of constructing
expanders on fiber is due to Alexander Engstr{\"o}m who used the zig-zag
product to obtain expanders~\cite{Engstrom2015}.  Our method is
different and yields for fixed $n\in\NN$ an expanding family for
$n\times n$ contingency tables where all row and column sums are
equal. Our adapted Markov basis can become arbitrarily large and a
priori it is not easy to draw a move from it uniformly at random. It
remains an interesting problem -- both from a combinatorial and
statistical side -- to understand the structure of the adapted Markov
basis and how one can draw from it efficiently.

\subsection*{Conventions and Notation}
The natural numbers are $\NN:=\{0,1,2,\ldots\}$. For any $n\in\NN$, we
set $[n]:=\{m\in\NN: 1\le m \le n\}$ and we use $\NN_{>
n}$ and $\NN_{\ge n}$ to denote the subsets of $\NN$ whose elements
are strictly greater and greater than $n$ respectively.
A graph is always undirected and can multiple
loops. Let $G=(V,E)$ be a graph. If there is $d\in\NN$ such that all the
nodes of $G$ are incident to $d$ edges, then $G$ is
\emph{$d$-regular}. The \emph{distance}
$d_G(v,w)$ of two nodes $v,w\in V$ is the number of edges in a
shortest path connecting $v$ and $w$. The \emph{diameter} $\diam{G}$
of $G$ is the maximal distance that appears between any pair of its
nodes. Here, it is assumed that all fiber-defining matrices $A\in\ZZ^{m\times d}$
fulfill $\ker_\ZZ(A)\cap\NN^d=\{0\}$. The affine semigroup
in $\ZZ^m$
generated by the column vectors of $A$ is denoted by $\cone{A}$. Let
$(a_i)_{i\in\NN}$ and $(b_i)_{i\in\NN}$ be two sequences in $\QQ$,
then $(a_i)_{i\in\NN}\in\mathcal{O}(b_i)_{i\in\NN}$ if there exist
$i_0\in\NN$ and $C\in\QQ_{>0}$ such that $|a_i|\le C\cdot|b_i|$ for
all $i\ge i_0$. Similarly, we define
$(a_i)_{i\in\NN}\in\Omega(b_i)_{i\in\NN}$ if there exist $i_0\in\NN$
and $C\in\QQ_{> 0}$ such that $|a_i|\ge C\cdot|b_i|$ for all $i\ge
i_0$. The sequence $(a_i)_{i\in\NN}$ is a subsequence of
$(b_i)_{i\in\NN}$ if there is a strongly increasing sequence
$(i_k)_{k\in\NN}$ in $\NN$ such that $a_{i_k}=b_k$ for all $k\in\NN$. 

\subsection*{Acknowledgements}
The author is very grateful to his supervisor Thomas Kahle for
uncountable comments and fruitful discussion and to Alexander Engström
for many hints on random walks on fiber graphs and families of
expanders. He is also thankful to Winfried Bruns for interesting
discussions about sequences in affine semigroups. He is supported
by the \href{https://www.studienstiftung.de}{German National
Academic Foundation} and
\href{https://www.ma.tum.de/TopMath/WebHomeEn}{TopMath}, a
graduate program of the
\href{https://www.elitenetzwerk.bayern.de}{Elite Network of
Bavaria}.

\section{Markov chains on fiber graphs}

Let $G=(V,E)$ be an undirected graph with $V=\{v_1,\ldots,v_n\}$. For
$v_i,v_j\in V$, let $A^G_{ij}$ be the number of edges in $E$
with endpoints $v_i$ and $v_j$, then the matrix $A^G\in\NN_{\ge
0}^{n\times n}$ is the \emph{adjacency matrix} of $G$. For $v\in V$,
let $\deg_G(v)$ be the number of edges incident to $v$ in $G$. The
\emph{simple walk} on $G$ has transition probabilities
\begin{equation*}
S^G_{ij}=\begin{cases}
\frac{A_{ij}^G}{\deg_G(v_i)},&\text{ if $\{v_i,v_j\}\in E$}\\
0,&\text{ otherwise}
\end{cases}.
\end{equation*}

The simple walk on $G$ comes along with a discrete-time Markov chain whose state
space is the node set $V$ of the graph~\cite{Boyd2004}. Let $\pi_0\in[0,1]^n$
be an initial distribution on $V$. For any $t\in\NN$, let
$\pi_t:=\pi_0\cdot(S^G)^t\in[0,1]^n$, then $\pi_t(i)$ is
the probability that the simple walk with initial distribution
$\pi_0$ is at $v_i$ at time $t$.
The Markov chain on $G$ is \emph{aperiodic} if for all $i\in[n]$,
$\gcd\{t\in\NN_{>0}: (S^G)^t_{i,i}>0\}=1$, \emph{symmetric} if $S^G$ is
symmetric, and \emph{irreducible} if for all $i,j\in[n]$ there exists
$t\in\NN$ such that $(S^G)^t_{i,j}>0$.
An aperiodic and irreducible Markov chain converges towards a unique
stationary distribution $\pi\in [0,1]^n$~\cite[Theorem~4.9]{Levin2008} and the second
largest eigenvalue modulus (SLEM) $\lambda$ of $S^G$, is a
measurement of the convergence rate~\cite[Section~3]{Hoory2006}:
$(\|\pi_t-\pi\|)_{t\in\NN}\in\mathcal{O}(\lambda^t)_{t\in\NN}$.

\begin{remark}\label{r:RegularSymmetric}
If $G$ is $d$-regular, then $S^G=\frac{1}{d}A^G$ and hence
$S^G$ is symmetric. If $G$ is also connected, then the simple walk is
irreducible and converges towards the uniform distribution
$\pi=\frac{1}{|V|}\cdot(1,\dots,1)^T\in[0,1]^{|V|}$ on $V$.
\end{remark}

The closer the second largest eigenvalue of a random walk is to $1$,
the slower the convergence to its stationary distribution. The next
definition states under which conditions we still have a polynomial
bound on the mixing time.

\begin{defn}\label{d:AsymptoticMixing}
For any $i\in\NN$, let $G_i=(V_i,E_i)$ be a graph and let $\lambda_i$ be
the second largest eigenvalue modulus of $S^{G_i}$. The sequence
$(G_i)_{i\in\NN}$ is \emph{rapidly mixing} if
there is a polynomial $p\in\QQ_{\ge 0}[t]$ such that for all $i\in\NN$,
$\lambda_i\le1-\frac{1}{p(\log|V_i|)}.$
The sequence $(G_i)_{i\in\NN}$ is an \emph{expander} if there exists
$\epsilon>0$ such that for all $i\in\NN$, $\lambda_i\le 1-\epsilon$.
\end{defn}

Expander graphs are highly demanded in computer science because of
their good mixing behaviour. Their name relates to the fact that their
\emph{edge-expansion} (Definition~\ref{d:EdgeExpansion}) can strictly
bounded away from zero (Proposition~\ref{p:EdgeExpansion}). The mixing
time of rapidly mixing Markov chains can be bounded by a polynomial in
the logarithm of the size of the state space (see
\cite[Section~2.3]{Sinclair1993} or \cite[Section~1.1.2]{Boyd2004})
and thus only logarithmically many nodes have to be visited by the
random walk to converge. The key player of this paper is the simple
walk on the following type of graph:

\begin{defn}\label{d:FiberGraph}
Let $\cF,\cM\subset\ZZ^d$ be finite sets. The \emph{fiber graph}
$\cF(\cM)$ is the graph on $\cF$ where two nodes $u,v\in\cF$ are
adjacent if $u-v\in\pm\cM$ and where every node $w\in\cF$ gets a loop
for every $m\in\pm\cM$ that satisfies $w+m\not\in\cF$.
\end{defn}

Recall that graphs can have multiple loops. The edges incident to a
node $v\in\cF$ correspond precisely to elements in $\pm\cM$ and thus
the graph $\cF(\cM)$ is $|\pm\cM|$-regular. If $0\in\cM$, then
$v-v\in\pm\cM$ and thus every node has at least one loop.  In order to
run irreducible Markov chains on fiber graphs, these graphs have to be
connected.

\begin{defn}\label{d:MarkovBasis}
Let $\cF,\cM\subset\ZZ^d$ be finite sets, then $\cM$ is a \emph{Markov
basis} for $\cF$ if the graph $\cF(\cM)$ is connected.  Let $\cI$ be
a set of indices, $d_i\in\NN$ be natural numbers and
$\cF_i\subset\ZZ^{d_i}$ and $\cM_i\subset\ZZ^{d_i}$ be finite sets for
any $i\in\cI$. A sequence $(\cM_i)_{i\in\cI}$ is a \emph{Markov basis}
for $(\cF_i)_{i\in\cI}$ if $\cM_i$ is a Markov basis for $\cF_i$ for
all $i\in\cI$. A finite set $\cM\subset\ZZ^d$ is a \emph{Markov basis}
for $(\cF_i)_{i\in\cI}$ with $\cF_i\subset\ZZ^d$ if $(\cM)_{i\in\NN}$
is a Markov basis for $(\cF_i)_{i\in\cI}$. 
\end{defn}

In many applications, $\cF\subset\ZZ^d$ is given implicit by
$\ZZ$-linear equations and inequalities and thus its complete
structure is unfeasible. In algebraic statistics for instance, $\cF$
equals $\fiber{A}{b}$ for a matrix $A\in\ZZ^{m\times d}$ and a
right-hand side $b\in\ZZ^m$~\cite{Drton2008}.  Note that our
general assumption $\ker_\ZZ(A)\cap\NN^d=\{0\}$ makes $\fiber{A}{b}$ finite
for all $b\in\cone{A}$. We thus call a set $\cM\subset\ZZ$ a
\emph{Markov basis} for $A$ if $\cM$ is a Markov basis for
$(\fiber{A}{b})_{b\in\cone{A}}$. 
If one can efficiently verify whether $v\in\ZZ^d$ is contained in
$\cF$ or not, as for $\cF=\fiber{A}{b}$, then it is possible to
explore $\cF$ with the simple walk using $\cM$ as follows: At a given
node $v\in\cF$, select uniformly an element $m\in\pm\cM$ and walk
along the edge given by $m\in\cM$, which either points to $v$ or to a
different node $v\neq v+m\in\cF$.

\begin{lemma}\label{l:ErgodicSimpleFiberWalks}
Let $\cF\subset\ZZ^d$ be a finite and non-empty set and
$\cM\subset\ZZ^d$ a Markov basis for $\cF$. The simple walk on
$\cF(\cM)$ is irreducible, aperiodic, symmetric, reversible, and its
stationary distribution is the uniform distribution on $\cF$.
\end{lemma}
\begin{proof}
The random walk is irreducible and symmetric since $\cF(\cM)$ is connected
and $|\pm\cM|$-regular (Remark~\ref{r:RegularSymmetric}). Thus, it suffices to show that $\cF(\cM)$ has
one aperiodic state to show that all states are aperiodic. Choose
$v\in\cF$ and $m\in\cM$ arbitrarily. Since~$\cF$ is finite, let
$\mu\in\NN$ be the largest natural number such that $v+\mu
m\in\cF$. Then $m$ cannot be applied on $v+\mu m$ and thus
$v+\mu m$ has a loop.  The reversibility follows immediately and
since the transition matrix of the simple walk is symmetric, the
uniform distribution is the unique stationary distribution. 
\end{proof}

\begin{remark}\label{r:MetropolisHastings}
The \emph{Metropolis-Hastings}-methodology allows to modify the
simple walk so that it converges to any given
probability distribution on $\cF$~\cite[Section~3]{Levin2008}.
\end{remark}

\section{Expanding in fixed dimension}\label{sec:ExpandingRay}
Let $A\in\ZZ^{m\times d}$ and $\cM\subset\ker_\ZZ(A)$ be a Markov
basis for $A$. In this section, we study the mixing behaviour of
$(\fibergraph{A}{b_i}{\cM})_{i\in\NN}$  for sequences
$(b_i)_{i\in\NN}$ in $\cone{A}$ that are almost rays
(Definition~\ref{d:DominatingRays}). Roughly speaking, our strategy is
to show that the sequence of second largest eigenvalues
$(\lambda_i)_{i\in\NN}$ of this graph sequence satisfies $\lambda_i\ge
1-\frac{C}{i}$ for a constant $C\in\QQ_{> 0}$ and sufficiently many
$i\in\NN$. This leads, together with an
assumption on the growth of the fiber
(Definition~\ref{d:MeaningfulParametrization}), to a slow
mixing result (Theorem~\ref{t:NotRapidlyMixing}).  Our
proof uses the well-known connection between the second largest
eigenvalue modulus of the simple walk and the \emph{edge-expansion} of
the underlying graph (Proposition~\ref{p:EdgeExpansion}).  Thus, our
goal is to bound the edge-expansion of
$(\fibergraph{A}{b_i}{\cM})_{i\in\NN}$ from above appropriately. Here,
we use a particular property of $(b_i)_{i\in\NN}$, namely that we can
translate a smaller fiber into a larger fiber
$u+\fiber{A}{b_i}\subseteq\fiber{A}{b_j}$. It is then left to count the number
of Markov moves that leave the subset $u+\fiber{A}{b_i}$ in
$\fibergraph{A}{b_j}{\cM}$, which is done by
Lemma~\ref{l:ExpansionOfBoundary} and Ehrhart's theory. To start with,
let us make precise the properties of sequences in~$\cone{A}$ that are
crucial in the proof of Theorem~\ref{t:NotRapidlyMixing}.

\begin{defn}
Let $A\in\ZZ^{m\times d}$ and let $(b_i)_{i\in\NN}$ be a sequence
in $\cone{A}$. 
The sequence $(b_i)_{i\in\NN}$ is a \emph{ray} in $\cone{A}$ if there
is $b\in\cone{A}$ such that $(b_i)_{i\in\NN}=(i\cdot b)_{i\in\NN}$.
\end{defn}

We need the following terminology for our next definition: For
$b\in\cone{A}$, the \emph{$\QQ$-relaxation} of $\fiber{A}{b}$ is the
polytope $\relaxfiber{A}{b}:=\{x\in\QQ_{\ge 0}^d: Ax=b\}$. 

\begin{defn}\label{d:DominatingRays}
Let $A\in\ZZ^{m\times d}$. A sequence $(b_i)_{i\in\NN}$ is
\emph{dominated}\index{sequence!dominated} if there exists
$b\in\cone{A}$ with $\dim(\relaxfiber{A}{b})>0$ such that
$b_i-i\cdot b\in\cone{A}$ for all $i\in\NN$ and if there is
$u\in\fiber{A}{b}$ and $w_i\in\fiber{A}{b_i-i\cdot b}$ with
$\supp(w_i)\subseteq\supp(u)$ for all $i\in\NN$.
\end{defn}

On the one hand, being dominated is a sufficient, though technical,
condition on $(b_i)_{i\in\NN}$ that is crucial in our proof of the
asymptotic growth of the second largest eigenvalue modulus of
$(\fibergraph{A}{b_i}{\cM})_{i\in\NN}$. The prime example of a
dominated sequence the reader should have in mind is a ray in the
semigroup~$\cone{A}$:

\begin{remark}\label{r:RaysVsDominatingRays}
The ray $(i\cdot b)_{i\in\NN}$ with $b\in\cone{A}$ and
$\dim(\relaxfiber{A}{b})>0$ is dominated by $b$.
\end{remark}

Dominated sequences appear, for instance, as subsequence of sequences
whose distance to the facets of $\cone{A}$ becomes arbitrarily large.
Let $H_A(b):=\min\{\text{dist}(b,F): F\textnormal{ facet of }\cone{A}\}$,
where $\text{dist}(b,F)\in\QQ_{\ge 0}$ denotes the distance between
$b$ and $F\subseteq\cone{A}$.

\begin{prop}\label{p:DistanceToFacets}
Let $A\in\ZZ^{m\times d}$ with non-trivial kernel.
Let $(b_i)_{i\in\NN}$ be a sequence in $\cone{A}$ with
$\limsup_{i\in\NN} H_A(b_i)=\infty$, then $(b_i)_{i\in\NN}$ has a
dominated subsequence.
\end{prop}
\begin{proof}
Let $a_1,\dots,a_d\in\ZZ^m$ be the columns of $A$ and let
$c:=a_1+\cdots+a_d$.
First, we show the following: For every $k\in\NN$ there exists
$m_k\in\NN$, such that any $b\in\cone{A}$ with $H_A(b)\ge m_k$
is contained in $k\cdot c+\cone{A}$. The set
$\cone{A}\setminus(k\cdot c+\cone{A})$ is contained in finitely many
hyperplanes parallel to the facets of $\cone{A}$. Hence, choosing
$m_k\in\NN$ large enough, every $b\in\cone{A}$ with
$H_A(b)\ge m_k$ cannot be in $\cone{A}\setminus(k\cdot c+\cone{A})$.
The statement of the lemma follows immediately because
$\limsup_{i\in\NN} H_A(b_i)=\infty$ implies that there is
$i_k\in\NN$ such that $H_A(b_{i_k})\ge m_k$. Hence, for all
$k\in\NN$, $b_{i_k}\in k\cdot c+\cone{A}$. In particular,
$(b_{i_k})_{k\in\NN}$ is dominated by $c$ since
$\fiber{A}{c}$ has an
element with full support and
since $\dim(\relaxfiber{A}{c})=\dim(\ker_\ZZ(A))>0$.
\end{proof}

\begin{remark}\label{r:ReverseOfFacetDistance}
The reverse of Proposition~\ref{p:DistanceToFacets} is not true. For
instance, take
the  matrix
$$A=\begin{bmatrix} 1 & 1 & 0 \\ 0 & 0 & 1\end{bmatrix}$$
and 
$b=(2,0)^T\in\ZZ^2$. The ray
$(i\cdot b)_{i\in\NN}$ is dominated since $\dim(\relaxfiber{A}{b})>0$. However,
since $\{i\cdot b: i\in\NN\}$ is contained in a facet of $\cone{A}$,
$H_A(i\cdot b)=0$
for all $i\in\NN$.
\end{remark}

To put hands on the second largest eigenvalue of the simple walk, we
use the following connecting piece between statistics and
combinatorics.

\begin{defn}\label{d:EdgeExpansion}
Let $G=(V,E)$ be a graph and $S\subseteq V$.  Then $E_G(S)\subseteq
E$ denotes the set of all edges with endpoints in $S$ and $V\setminus S$.
The \emph{edge-expansion} of $G$ is
\begin{equation*}
h(G):=\min\left\{\frac{|E_G(S)|}{|S|}: S\subset
V,0<2|S|\le|V|\right\}.
\end{equation*}
\end{defn}

\begin{remark}\label{r:CheegerConstant}
The invariant $h(G)$ has many names in the literature, like
\emph{Cheeger constant}~\cite{Cheeger1969} or \emph{isoperimetric
number}~\cite{Mohar1989}.  Also, the \emph{conductance} $\Phi$ of the
simple walk on a $d$-regular graph $G$ fulfills $\Phi\cdot
d=h(G)$~\cite{sinclair1989}.
\end{remark}

\begin{prop}\label{p:EdgeExpansion}
Let $G=(V,E)$ be a connected and $d$-regular graph 
and let $\lambda$ be the second largest eigenvalue modulus of $S^G$, then
$\lambda\ge 1- \frac{2}{d}\cdot h(G)$.
\end{prop}
\begin{proof}
This is \cite[Theorem 4.11]{Hoory2006}.
\end{proof}

\begin{example}\label{ex:NoExpander}
For any $d\in\NN$, let $A_d=(1,\ldots,1)\in\ZZ^{1\times d}$ and let
$e_k\in\ZZ^d$ be the $k$-th unit vector of $\ZZ^d$, then the set
$\cM_d:=\{e_1-e_k: 2\le k\le d\}$ is a Markov basis for $A_d$. The
graph $\fibergraph{A_2}{i}{\cM_2}$ is isomorphic to the path graph on
$[i+1]$ for any $i\in\NN$ and hence its edge-expansion is
$\frac{2}{i+1}$ if $i$ is odd and $\frac{2}{i}$ when $i$ is
even~\cite[Section~2]{Mohar1989}. Since $|\pm\cM_2|=2$, the second largest
eigenvalue modulus $\lambda_i$ of the simple walk on
$\fibergraph{A_2}{i}{\cM_2}$ satisfies $\lambda_i\ge 1-\frac{1}{i}$
by Proposition~\ref{p:EdgeExpansion}. Hence, the
sequence $(\fibergraph{A_2}{i}{\cM_2})_{i\in\NN}$ is neither an
expander and because of $\log|\fiber{A_2}{i}|=\log(i+1)$ nor rapidly
mixing.
\end{example}

The edge-expansion of a graph can be bounded from above by dividing
the number of edges leaving a fixed subset by the size of this
particular subset. For certain subsets of fibers, it is possible to
give a description of the nodes that are incident to edges which leave this
set. Intuitively, those nodes lie on the \emph{boundary} of this set.

\begin{defn}\label{d:Boundary}
Let $A\in\ZZ^{m\times d}$, $b\in\cone{A}$, and
$\cM\subset\ker_\ZZ(A)$. For $u\in\NN^d$, the \emph{$u$-boundary of
$\fiber{A}{b}$} is
$\boundaryDir{\cM}{u}{\fiber{A}{b}}:=\{v\in u+\fiber{A}{b}: \exists m\in\pm\cM:
v+m\in\NN^d\setminus(u+\fiber{A}{b})\}$.
\end{defn}

\begin{figure}[htbp]
\begin{minipage}{0.33\textwidth}
\centering
\begin{tikzpicture}[xscale=0.5,yscale=0.5]

\fill[gray!20!white] (0,0) -- (1.5,3) -- (3,0);

\foreach \number in {0,1,2,3,4,5,6}{
\node[draw,circle,fill=black,inner sep=0pt,minimum size=4pt] () at (\number,0) {};
}

\foreach \number in {0,1,2,3,4,5}{
\node[draw,circle,fill=black,inner sep=0pt,minimum size=4pt] () at (0.5+\number,1) {};
}

\foreach \number in {0,1,2,3,4}{
\node[draw,circle,fill=black,inner sep=0pt,minimum size=4pt] () at (1+\number,2) {};
}

\foreach \number in {0,1,2,3}{
\node[draw,circle,fill=black,inner sep=0pt,minimum size=4pt] () at (1.5+\number,3) {};
}

\foreach \number in {0,1,2}{
\node[draw,circle,fill=black,inner sep=0pt,minimum size=4pt] () at (2+\number,4) {};
}

\foreach \number in {0,1}{
\node[draw,circle,fill=black,inner sep=0pt,minimum size=4pt] () at (2.5+\number,5) {};
}

\foreach \number in {0}{
\node[draw,circle,fill=black,inner sep=0pt,minimum size=4pt] () at (3+\number,6) {};
}

\draw[dotted,thick] (0,0)  -- (6,0) node () {};
\draw[dotted,thick] (0.5,1)  -- (5.5,1) node () {};
\draw[dotted,thick] (1,2)  -- (5,2) node () {};
\draw[dotted,thick] (1.5,3)  -- (4.5,3) node () {};
\draw[dotted,thick] (2,4)  -- (4,4) node () {};
\draw[dotted,thick] (2.5,5)  -- (3.5,5) node () {};

\draw[dotted,thick] (0,0)  -- (3,6) node () {};
\draw[dotted,thick] (1,0)  -- (3.5,5) node () {};
\draw[dotted,thick] (2,0)  -- (4,4) node () {};
\draw[dotted,thick] (3,0)  -- (4.5,3) node () {};
\draw[dotted,thick] (4,0)  -- (5,2) node () {};
\draw[dotted,thick] (5,0)  -- (5.5,1) node () {};

\draw[thick] (3,0)  -- (4,0) node () {};
\draw[thick] (3,0)  -- (3.5,1) node () {};
\node[draw,circle,fill=white,inner sep=0pt,minimum size=4pt] () at (3,0) {};

\draw[thick] (2.5,1)  -- (3.5,1) node () {};
\draw[thick] (2.5,1)  -- (3,2) node () {};
\node[draw,circle,fill=white,inner sep=0pt,minimum size=4pt] () at (2.5,1) {};

\draw[thick] (2,2)  -- (3,2) node () {};
\draw[thick] (2,2)  -- (2.5,3) node () {};
\node[draw,circle,fill=white,inner sep=0pt,minimum size=4pt] () at (2,2) {};

\draw[thick] (1.5,3)  -- (2.5,3) node () {};
\draw[thick] (1.5,3)  -- (2,4) node () {};
\node[draw,circle,fill=white,inner sep=0pt,minimum size=4pt] () at (1.5,3) {};

\end{tikzpicture}
\end{minipage}\hfill
\begin{minipage}{0.33\textwidth}
\centering
\begin{tikzpicture}[xscale=0.5,yscale=0.5]

\fill[gray!20!white] (0,0) -- (1.5,3) -- (3,0);

\foreach \number in {0,1,2,3,4,5,6}{
\node[draw,circle,fill=black,inner sep=0pt,minimum size=4pt] () at (\number,0) {};
}

\foreach \number in {0,1,2,3,4,5}{
\node[draw,circle,fill=black,inner sep=0pt,minimum size=4pt] () at (0.5+\number,1) {};
}

\foreach \number in {0,1,2,3,4}{
\node[draw,circle,fill=black,inner sep=0pt,minimum size=4pt] () at (1+\number,2) {};
}

\foreach \number in {0,1,2,3}{
\node[draw,circle,fill=black,inner sep=0pt,minimum size=4pt] () at (1.5+\number,3) {};
}

\foreach \number in {0,1,2}{
\node[draw,circle,fill=black,inner sep=0pt,minimum size=4pt] () at (2+\number,4) {};
}

\foreach \number in {0,1}{
\node[draw,circle,fill=black,inner sep=0pt,minimum size=4pt] () at (2.5+\number,5) {};
}

\foreach \number in {0}{
\node[draw,circle,fill=black,inner sep=0pt,minimum size=4pt] () at (3+\number,6) {};
}

\draw[dotted,thick] (0,0)  -- (6,0) node () {};
\draw[dotted,thick] (0.5,1)  -- (5.5,1) node () {};
\draw[dotted,thick] (1,2)  -- (5,2) node () {};
\draw[dotted,thick] (1.5,3)  -- (4.5,3) node () {};
\draw[dotted,thick] (2,4)  -- (4,4) node () {};
\draw[dotted,thick] (2.5,5)  -- (3.5,5) node () {};

\draw[dotted,thick] (0,0)  -- (3,6) node () {};
\draw[dotted,thick] (1,0)  -- (3.5,5) node () {};
\draw[dotted,thick] (2,0)  -- (4,4) node () {};
\draw[dotted,thick] (3,0)  -- (4.5,3) node () {};
\draw[dotted,thick] (4,0)  -- (5,2) node () {};
\draw[dotted,thick] (5,0)  -- (5.5,1) node () {};

\draw[thick] (3,0)  -- (4,0) node () {};
\draw[thick] (3,0)  -- (3.5,1) node () {};
\node[draw,circle,fill=white,inner sep=0pt,minimum size=4pt] () at (3,0) {};

\draw[thick] (2.5,1)  -- (3.5,1) node () {};
\draw[thick] (2.5,1)  -- (3,2) node () {};
\node[draw,circle,fill=white,inner sep=0pt,minimum size=4pt] () at (2.5,1) {};

\draw[thick] (2,2)  -- (3,2) node () {};
\draw[thick] (2,2)  -- (2.5,3) node () {};
\node[draw,circle,fill=white,inner sep=0pt,minimum size=4pt] () at (2,2) {};

\draw[thick] (1.5,3)  -- (2.5,3) node () {};
\draw[thick] (1.5,3)  -- (2,4) node () {};
\node[draw,circle,fill=white,inner sep=0pt,minimum size=4pt] () at (1.5,3) {};

\draw[thick] (2,0)  to [bend right=30] (4,0) node () {};
\draw[thick] (2,0)  to [bend left=30] (3,2) node () {};
\node[draw,circle,fill=white,inner sep=0pt,minimum size=4pt] () at (2,0) {};

\draw[thick] (1.5,1)  to [bend right=30] (3.5,1) node () {};
\draw[thick] (1.5,1)  to [bend left=30] (2.5,3) node () {};
\node[draw,circle,fill=white,inner sep=0pt,minimum size=4pt] () at (1.5,1) {};

\draw[thick] (1,2)  to [bend right=30] (3,2) node () {};
\draw[thick] (1,2)  to [bend left=30] (2,4) node () {};
\node[draw,circle,fill=white,inner sep=0pt,minimum size=4pt] () at (1,2) {};

\end{tikzpicture}
\end{minipage}\hfill
\begin{minipage}{0.33\textwidth}
\centering
\begin{tikzpicture}[xscale=0.5,yscale=0.5]

\fill[gray!20!white] (1.5,1) -- (4.5,1) -- (3,4);

\foreach \number in {0,1,2,3,4,5,6}{
\node[draw,circle,fill=black,inner sep=0pt,minimum size=4pt] () at (\number,0) {};
}

\foreach \number in {0,1,2,3,4,5}{
\node[draw,circle,fill=black,inner sep=0pt,minimum size=4pt] () at (0.5+\number,1) {};
}

\foreach \number in {0,1,2,3,4}{
\node[draw,circle,fill=black,inner sep=0pt,minimum size=4pt] () at (1+\number,2) {};
}

\foreach \number in {0,1,2,3}{
\node[draw,circle,fill=black,inner sep=0pt,minimum size=4pt] () at (1.5+\number,3) {};
}

\foreach \number in {0,1,2}{
\node[draw,circle,fill=black,inner sep=0pt,minimum size=4pt] () at (2+\number,4) {};
}

\foreach \number in {0,1}{
\node[draw,circle,fill=black,inner sep=0pt,minimum size=4pt] () at (2.5+\number,5) {};
}

\foreach \number in {0}{
\node[draw,circle,fill=black,inner sep=0pt,minimum size=4pt] () at (3+\number,6) {};
}

\draw[dotted,thick] (0,0)  -- (6,0) node () {};
\draw[dotted,thick] (0.5,1)  -- (5.5,1) node () {};
\draw[dotted,thick] (1,2)  -- (5,2) node () {};
\draw[dotted,thick] (1.5,3)  -- (4.5,3) node () {};
\draw[dotted,thick] (2,4)  -- (4,4) node () {};
\draw[dotted,thick] (2.5,5)  -- (3.5,5) node () {};

\draw[dotted,thick] (0,0)  -- (3,6) node () {};
\draw[dotted,thick] (1,0)  -- (3.5,5) node () {};
\draw[dotted,thick] (2,0)  -- (4,4) node () {};
\draw[dotted,thick] (3,0)  -- (4.5,3) node () {};
\draw[dotted,thick] (4,0)  -- (5,2) node () {};
\draw[dotted,thick] (5,0)  -- (5.5,1) node () {};

\draw[thick] (3,4)  -- (4,4) node () {};
\draw[thick] (3,4)  -- (2,4) node () {};
\draw[thick] (3,4)  -- (3.5,5) node () {};
\node[draw,circle,fill=white,inner sep=0pt,minimum size=4pt] () at (3,4) {};

\draw[thick] (2.5,3)  -- (1.5,3) node () {};
\node[draw,circle,fill=white,inner sep=0pt,minimum size=4pt] () at (2.5,3) {};

\draw[thick] (2,2)  -- (1,2) node () {};
\node[draw,circle,fill=white,inner sep=0pt,minimum size=4pt] () at (2,2) {};

\draw[thick] (1.5,1)  -- (0.5,1) node () {};
\draw[thick] (1.5,1)  -- (1,0) node () {};
\node[draw,circle,fill=white,inner sep=0pt,minimum size=4pt] () at (1.5,1) {};

\draw[thick] (2.5,1)  -- (2,0) node () {};
\node[draw,circle,fill=white,inner sep=0pt,minimum size=4pt] () at (2.5,1) {};

\draw[thick] (3.5,1)  -- (3,0) node () {};
\node[draw,circle,fill=white,inner sep=0pt,minimum size=4pt] () at (3.5,1) {};

\draw[thick] (4.5,1)  -- (4,0) node () {};
\draw[thick] (4.5,1)  -- (5.5,1) node () {};
\draw[thick] (4.5,1)  -- (5,2) node () {};
\node[draw,circle,fill=white,inner sep=0pt,minimum size=4pt] () at (4.5,1) {};

\draw[thick] (4,2)  -- (5,2) node () {};
\draw[thick] (4,2)  -- (4.5,3) node () {};
\node[draw,circle,fill=white,inner sep=0pt,minimum size=4pt] () at (4,2) {};

\draw[thick] (3.5,3)  -- (4.5,3) node () {};
\draw[thick] (3.5,3)  -- (4,4) node () {};
\node[draw,circle,fill=white,inner sep=0pt,minimum size=4pt] () at (3.5,3) {};

\end{tikzpicture}
\end{minipage}\hfill
\caption{Let $A_3$ be as in Example~\ref{ex:NoExpander}. The white points represent the nodes of the sets
$\boundaryDir{\cM_3}{(3,0,0)^T}{\fiber{A_3}{3}}$,
$\boundaryDir{\cM_3\cup 2\cdot\cM_3}{(3,0,0)^T}{\fiber{A_3}{3}}$, and
$\boundaryDir{\cM_3}{(1,1,1)^T}{\fiber{A_3}{3}}$ in
$\fiber{A_3}{6}$ (from left to right).}\label{f:Boundary}
\end{figure}

Figure~\ref{f:Boundary} justifies in a way that
$\boundaryDir{\cM}{u}{\fiber{A}{b}}$ can indeed be regarded as a
boundary. With this, the number of outgoing edges in a translated fiber
$u+\fiber{A}{b}$ within a larger fiber can be bounded from above:

\begin{lemma}\label{l:ExpansionOfBoundary}
Let $A\in\ZZ^{m\times d}$ and $b,b'\in\cone{A}$ with
$2|\fiber{A}{b}|\le|\fiber{A}{b'+b}|$.  Then for any finite set
$\cM\subset\ker_\ZZ(A)$ and all $u\in\fiber{A}{b'}$,
\begin{equation*}
h(\fibergraph{A}{b'+b}{\cM})\le\frac{2|\cM|\cdot|\boundaryDir{\cM}{u}{\fiber{A}{b}}|}{|\fiber{A}{b}|}.
\end{equation*}
\end{lemma}
\begin{proof}
By assumption, $u+\fiber{A}{b}\subset\fiber{A}{b'+b}$ and since
$2|u+\fiber{A}{b}|=2|\fiber{A}{b}|\le|\fiber{A}{b'+b}|$,
\begin{equation*}
h(\fibergraph{A}{b'+b}{\cM})\le
\frac{|E_{\fibergraph{A}{b'+b}{\cM}}(u+\fiber{A}{b})|}{|u+\fiber{A}{b}|}.
\end{equation*}
The edges leaving the set $u+\fiber{A}{b}$ in
$\fibergraph{A}{b'+b}{\cM}\subset\NN^d$ are precisely
those with endpoints in
$\boundaryDir{\cM}{u}{\fiber{A}{b}}$.
Every node of $\fibergraph{A}{b'+b}{\cM}$ has at most $|\pm\cM|$ incident
edges and hence $|E_{\fibergraph{A}{b'+b}{\cM}}(u+\fiber{A}{b})|$ is bounded from above by
$2|\cM|\cdot|\boundaryDir{\cM}{u}{\fiber{A}{b}}|$. 
\end{proof}

The size of the entries in a Markov basis is crucial to determine the
size of the boundary. The larger those entries are, the more nodes
are in the boundary (Lemma~\ref{l:BoundaryMultipleMoves}) since more
nodes in the shifted fiber $u+\fiber{A}{b}$ are adjacent to
nodes outside of $u+\fiber{A}{b}$. The next definition should not
be mixed up with the \emph{Markov complexity}~\cite{Santos2003}.

\begin{defn}
The \emph{complexity} of a finite set $\cM\subset\ZZ^d$ is 
$\cC(\cM):=\max_{m\in\cM}\|m\|_{\infty}$.
\end{defn}

\begin{lemma}\label{l:BoundaryMultipleMoves}
Let $A\in\ZZ^{m\times d}$, let 
$\cM\subset\ker_\ZZ(A)$ be a finite set, and let $b\in\cone{A}$. 
Then for all $u\in\NN^d$,
\begin{equation*}
\boundaryDir{\cM}{u}{\fiber{A}{b}}\subseteq
u+\bigcup_{j\in\supp(u)}\bigcup_{r=0}^{\cC(\cM)}\{w\in\fiber{A}{b}:w_j=r\}.
\end{equation*}
\end{lemma}
\begin{proof}
Let $v\in\boundaryDir{\cM}{u}{\fiber{A}{b}}$, then there is
$m\in\pm\cM$ such that $v+m\in\NN^d$, but $v+m\not\in
u+\fiber{A}{b}$. Since $v\in u+\fiber{A}{b}$, there is
$w\in\fiber{A}{b}$ such that $v=u+w$. The vector $w+m$ must have a
negative entry, since otherwise $w+m\in\NN^d$, that is
$w+m\in\fiber{A}{b}$ which implies $v+m=u+w+m\in u+\fiber{A}{b}$.
Hence, there is $j\in[d]$ such that $(w+m)_j<0$. 
Suppose $j\not\in\supp(u)$. Then $(u+w+m)_j=(w+m)_j<0$, which contradicts
$u+w+m=v+m\in\NN^d$. Thus, $j\in\supp(u)$ and
$w_j<-m_j$. Since that means $w_r\le\cC(\cM)$, the statement follows.
\end{proof}

Lemma~\ref{l:ExpansionOfBoundary} allows to measure edge-expansion by
essentially comparing the growth of fibers with the growth of their
boundary.  The idea is to show that the boundary grows asymptotically
slower than the fiber itself.  Counting the number of integer points
in a polytope is the subject of Ehrhart theory \cite{Ehrhart1977}.  Let
$P\subset\QQ^d$ be a polytope and consider the map $L_{P}:\NN\to\NN$ which
counts the integer points in the $i$-th dilation $iP:=\{i\cdot
x\in\QQ^d: x\in P\}$, i.e.
\begin{equation*}
L_{P}(i):=|(iP)\cap\ZZ^d|.
\end{equation*}
According to Ehrhart's theorem (cf. \cite[Theorem~3.23]{Beck2007}), $L_P$ is a
\emph{quasi-polynomial} of degree $r:=\dim(P)$, that is there exist
periodic functions $c_0,\ldots,c_r:\NN\to\ZZ$ with integral periods such that
\begin{equation*}
L_P(t)=c_r(t)t^r+c_{r-1}(t)t^{r-1}+ \ldots + c_0(t)
\end{equation*}
with $c_r$ not identically zero. Here, the dimension of a set is the
dimension of its affine space. This applies to rays in affine
semigroups: Since for any
$i\in\NN$, the integer points of $\relaxfiber{A}{ib}$ are precisely
the elements of $\fiber{A}{ib}$,
$L_{\relaxfiber{A}{b}}(i)=|\fiber{A}{ib}|$ for all $i\in\NN$ and hence
$|\fiber{A}{ib}|$ grows in $i$ (quasi-)polynomial of degree
$\dim(\relaxfiber{A}{b})$.

\begin{remark}\label{r:TotallyUnimodular}
For any integer matrix $A$ and $b\in\cone{A}$,
$\dim(\relaxfiber{A}{b})\ge\dim(\fiber{A}{b})$. In particular, if
$\dim(\fiber{A}{b})>0$, then $(|\fiber{A}{ib}|)_{i\in\NN}$ is
unbounded. If $A$ is totally unimodular, then
$\relaxfiber{A}{b}=\conv(\fiber{A}{b})$ and hence the dimensions of
$\relaxfiber{A}{b}$ and $\fiber{A}{b}$ coincide.
\end{remark}

The sets appearing in Lemma~\ref{l:BoundaryMultipleMoves} are not
precisely dilates of polytopes and Ehrhart theory does not apply
directly. Nevertheless, their growth can be bounded as well in terms
of their dimension.

\begin{lemma}\label{l:HyperplanesAndPolytopes}
Let $A\in\ZZ^{m\times d}$, $b\in\ZZ^m$, and fix integers
$j\in[d]$ and $l\in\NN$. If for all $i\in\NN_{>0}$,
$\relaxfiber{A}{ib}$ is not completely
contained in the hyperplane $H:=\{x\in\QQ^d: x_j=l\}$, then there is
$C\in\NN$ such that the number of integer points in
$\relaxfiber{A}{ib}\cap H$ is bounded from above by $ C\cdot
i^{\dim(\relaxfiber{A}{b})-1}$. 
\end{lemma}
\begin{proof}
Write $P:=\relaxfiber{A}{b}$ and $r:=\dim(P)$.  For $i$ large
enough, the dimension of $(iP)\cap H$ stabilizes, i.e. there are
$r',N\in\NN$ such that $r':=\dim(iP\cap H)$ for all $i\ge N$. The
affine space of $iP\cap H$ is completely contained in $H$ whereas the
affine space of $iP$ has elements outside of $H$. That implies 
$r'<r$.  Let $A=(a_1,\ldots,a_d)$ and $A'$ be 
submatrix of $A$ omitting the $j$-th column, then the (bijective) projection
of $iP\cap H$ onto all coordinates different from $j$ is
\begin{equation*}
Q_i:=\{x\in\QQ^{d-1}: A'x=ib-la_j\}.
\end{equation*}
By \cite[Proposition~1]{Verdoolaege2007}, there exists finitely many
sets $C_1,\ldots,C_k$ covering $\NN$ such that for $i\in C_j$, the
number of integer points in $Q_i$ is a quasi-polynomial of degree
$r'$. 
\end{proof}

\begin{lemma}\label{l:GrowthOfQuasiPoly}
Let $p(t)=\sum_{s=0}^rc_s(t)t^s$ be a quasi-polynomial of degree $r>0$
and let $k\in\NN$ such that $c_r(k)>0$. There exists $n\in\NN_{>0}$
and $N\in\NN$ such that for all $i\in k+n\cdot\NN$ with $i\ge N$,
$2p(i)<p(i+ni)$.
\end{lemma}
\begin{proof}
Let $n\ge 2$ such that $c_r(i+ni)=c_r(i)$ for all $i\in\NN$ (i.e. if
$c_r$ is not a constant, let $n\ge 2$ be the period of $c_r$). For all
$i\in k+n\cdot\NN$, $c_r(i+ni)=c_r(i)=c_r(k)>0$ and
\begin{equation*}
p(i+ni)-2\cdot
p(i)=c_r(k)\left(
(1+n)^r-2\right)i^r+\sum_{s=0}^{r-1}\left(c_s(i+ni)(1+n)^s-2c_s(i)\right)i^s.
\end{equation*}
The sum in the term on the right-hand side of this equation is a
quasi-polynomial of degree at most $r-1$ and the left term on the
right-hand side is a polynomial of degree $r>0$ whose leading
coefficient is positive due to $n\ge 2$ and $r>0$. Thus, there is
$N\in\NN$ such that for all $i\in k+n\cdot\NN$ with $i\ge N$,
\begin{equation*}
c_r(k)\left(
(1+n)^r-2\right)i^r>
-\sum_{s=0}^{r-1}\left(c_s(i+in)(1+n)^s-2c_s(k)\right)i^s,
\end{equation*}
that is $2p(i)<p(i+ni)$.
\end{proof}

The growth of the state space needs to be compared with the growth of
the second largest eigenvalue modulus to disprove rapid mixing, as
demonstrated in Example~\ref{ex:NoExpander}. With the following
property of a sequences $(b_i)_{i\in\NN}$ in $\cone{A}$, however, a
lower bound the second largest eigenvalue in terms of the parameter
$(i)_{i\in\NN}$ instead of $(|\fiber{A}{b_i}|)_{i\in\NN}$ suffices:

\begin{defn}\label{d:MeaningfulParametrization}
A sequence $(b_i)_{i\in\NN}$ in $\cone{A}$ has a \emph{meaningful
parametrization} if there exists a polynomial $q\in\QQ[t]$ such that
$|\fiber{A}{b_i}|\le q(i)$ for all $i\in\NN$.
\end{defn}

\begin{example}
Consider $A_2$ from Example~\ref{ex:NoExpander}, then
$|\fiber{A_2}{i}|=i+1$. The sequence $(2^i)_{i\in\NN}$ in
$\cone{A_2}=\NN$ is not meaningfully parametrized. However, it is a
subsequence of the sequence $(i)_{i\in\NN}$, which has a meaningful
parametrization.
\end{example}

\begin{prop}\label{p:MeaningfulParametrization}
Let $A\in\ZZ^{m\times d}$ and let $(b_i)_{i\in\NN}$ be a sequence in
$\cone{A}$ satisfying
$(\|b_i\|)_{i\in\NN}\in\mathcal{O}(i^r)_{i\in\NN}$ for some $r\in\NN$.
Then $(b_i)_{i\in\NN}$ has a meaningful parametrization.
\end{prop}
\begin{proof}
Denote by $a_1,\dots,a_m\in\ZZ^d$ the rows of $A$. Since
$\ker_\ZZ(A)\cap\NN^d=\{0\}$, there exist coefficients
$\lambda_1,\dots,\lambda_m\in\QQ$ such that
$w:=\sum_{i=1}^m\lambda_ia_i\in\QQ^d_{>0}$. In particular, for any
$b\in\NN A$ and for any $u\in\fiber{A}{b}$, 
$\|u\|_{\infty}\cdot\min_{i\in[d]}w_i\le w^Tu\le
m\cdot\|\lambda\|_\infty\cdot\|b\|_\infty$. Thus, 
\begin{equation*}
|\fiber{A}{b}|\le\left(\frac{m\cdot\|\lambda\|_\infty\|b\|_\infty}{\min_{i\in[d]}w_i}\right)^d.
\end{equation*} 
Hence, if $\|b_i\|\le C\cdot i^r$ for all $i\in\NN$,
then $(b_i)_{i\in\NN}$ has a meaningful parametrization.
\end{proof}

For sequences with a meaningful parametrization, it suffices to bound
the edge-expansion appropriately from above to show a slow mixing
behaviour.

\begin{lemma}\label{l:MeaningfulLemma}
Let $A\in\ZZ^{m\times d}$, $\cM\subset\ker_\ZZ(A)$ be a finite set, and
let $(b_i)_{i\in\NN}$ be a sequence in $\cone{A}$ with meaningful
parametrization. If there is an infinite subset $\cI\subseteq\NN$ and
$C\in\NN$ such that $h(\fibergraph{A}{b_i}{\cM})\le\frac{C}{i}$ for
all $i\in\cI$, then $(\fibergraph{A}{b_i}{\cM})_{i\in\NN}$ is not
rapidly mixing.
\end{lemma}
\begin{proof}
Let $\lambda_i\in[0,1]$ be the second largest eigenvalue modulus of
the simple walk on $\fibergraph{A}{b_i}{\cM}$ and assume that the
sequence mixes rapidly. Then there exists a polynomial $p\in\QQ_{\ge
0}[t]$
such that for all $i\in\cI$,
\begin{equation*}
1-\frac{1}{p(\log|\fiber{A}{b_i}|)}\ge \lambda_i\ge 1-\frac{C}{|\cM|\cdot i}
\end{equation*}
where we have used the assumption on the edge-expansion and
Proposition~\ref{p:EdgeExpansion} to obtain the lower bound. This
implies that for all $i\in\cI$,
$\frac{1}{i}\cdot p(\log|\fiber{A}{b_i}|)\ge\frac{|\cM|}{C}$. However,
since the parametrization is meaningful, there exists a polynomial
$q\in\QQ[t]$ such that $|\fiber{A}{b_i}|\le q(i)$ and thus
$p(\log|\fiber{A}{b_i}|)\le p(\log q(i))$ for $i$ sufficiently large, which gives a
contradiction since $\cI$ is an infinite subset and hence unbounded.
\end{proof}
We are now ready to prove our main theorem: 
\begin{proof}[Proof of Theorem~\ref{t:NotRapidlyMixing}]
Since $(b_i)_{i\in\NN}$ is dominated, there is $b\in\cone{A}$ with
$\dim(\relaxfiber{A}{b})>0$ such that $b_i':=b_i-i\cdot b\in\cone{A}$
for all $i\in\NN$. Moreover, there is $u\in\fiber{A}{b}$ and a
sequence $(w_i)_{i\in\NN}$ in $\NN^d$ such that for all $i\in\NN$,
$w_i\in\fiber{A}{b_i'}$ and $\supp(w_i)\subseteq\supp(u)$. 
Clearly, it suffices to show the theorem for the subsequence
$(b_{(\cC(\cM)+1)i})_{i\in\NN}$ which inherits the properties of
being dominated and being meaningfully parametrized from $(b_i)_{i\in\NN}$
due to the linear re-parametrization. 
Thus, we replace $b_i$ with $b_{(\cC(\cM)+1)i}$, $b$
with $(\cC(\cM)+1)\cdot b$, and $b_i'$ with
$b_{(\cC(\cM)+1)i}'$.  Additionally, we replace $w_i$ with
$w_{(\cC(\cM)+1)i}$ and $u$ with $(\cC(\cM)+1)\cdot u$,
which does not change the support of $u$.
After these changes, we have $u_i>\cC(\cM)$ for all
$i\in\supp(u)$, which is needed later in the proof.
The Ehrhart quasi-polynomial
$L_{\relaxfiber{A}{b}}$ has degree $r:=\dim(\relaxfiber{A}{b})$ and by
the definition of being dominated, $r>0$. 
Write $L_{\relaxfiber{A}{b}}(i)=\sum_{s=0}^r c_s(i)i^s$ with $c_r$ not
identically zero. 
Since $L_{\relaxfiber{A}{b}}(i)=|\fiber{A}{ib}|>
0$, there exists $k\in\NN$ such that $c_r(k)>0$. By Lemma
\ref{l:GrowthOfQuasiPoly}, there exists $n\in\NN_{>0}$ and $N\in\NN$
such that $2|\fiber{A}{ib}|\le|\fiber{A}{(i+ni)b}|$ for all $i\in
(k+n\cdot\NN)\cap\NN_{\ge N}=:\cI$. By the choice of $w_i$ and
$u$, $A\cdot(w_{i+ni}+ ni\cdot u)=b_{i+ni}'+ni\cdot b=b_{i+ni}-ib$ for
all $i\in\cI$ and hence $w_{i+ni}+ni\cdot
u+\fiber{A}{ib}\subsetneq\fiber{A}{b_{i+ni}}$. In particular, for any
$i\in\cI$
\begin{equation*}
2|\fiber{A}{ib}|\le|\fiber{A}{(i+ni)b}|=|w_{i+ni}+\fiber{A}{(i+ni)b}|\le|\fiber{A}{b_{i+ni}}|.
\end{equation*}
For any $i\in\cI$, set $u_i:=w_{i+ni}+ni\cdot u$, then 
Lemma~\ref{l:BoundaryMultipleMoves} gives  
\begin{equation}\label{equ:BoundBoundary}
|\boundaryDir{\cM}{u_i}{\fiber{A}{ib}}|\le
\sum_{j\in\supp(u_i)}\sum_{l=0}^{\cC(\cM)}|\{v\in\fiber{A}{ib}:v_j=l\}|.
\end{equation}
Since $2|\fiber{A}{ib}|\le|\fiber{A}{b_{i+ni}}|$
and $u_{i}\in\fiber{A}{b'_{i+ni}+ni\cdot b}$ for all $i\in\cI$,
an application of Lemma~\ref{l:ExpansionOfBoundary} yields the upper
bound on the edge-expansion of the graph
$\fibergraph{A}{b_{i+ni}}{\cM}$:
\begin{equation*}
h(\fibergraph{A}{b_{i+ni}}{\cM})\le
\frac{2|\cM|\cdot|\boundaryDir{\cM}{u_i}{\fiber{A}{ib}}|}{|\fiber{A}{ib}|}\le
2|\cM|\cdot\frac{\sum_{j\in\supp(u)}\sum_{l=0}^{\cC(\cM)}|\{w\in\fiber{A}{ib}:w_j=l\}|}{|\fiber{A}{ib}|},
\end{equation*}
where equation \eqref{equ:BoundBoundary} and
$\supp(u_i)\subseteq\supp(u)$ was used in the first and second
inequality respectively.
For any $j\in\supp(u)$ and $l\in[\cC(\cM)]\cup\{0\}$, let $H_{j,l}=\{x\in\QQ^d:
x_j=l\}$ be a hyperplane in $\QQ^d$, then for all $i\in\cI$, the
integer points in $(i\cdot\relaxfiber{A}{b})\cap H_{j,l}$ are
precisely $L_{j,l}(i):=|\{w\in\fiber{A}{ib}: w_j=l\}|$.
Since $u_j>\cC(\cM)$ for all $j\in\supp(u)$, the vector $i\cdot u\in
\relaxfiber{A}{ib}=i\cdot\relaxfiber{A}{b}$ is not contained in
$(i\cdot\relaxfiber{A}{b})\cap H_{j,l}$ for all $l\in[\cC(\cM)]\cup\{0\}$ and
all $i\in\cI$.  Lemma~\ref{l:HyperplanesAndPolytopes} then implies
that for all $j\in\supp(u)$ and $l\in[\cC(\cM)]\cup\{0\}$, there is
a constant $C_{j,l}\in\NN$ such that $L_{j,l}(i)\le C_{j,l}\cdot i^{r-1}$ for all
$i\in\NN$. Let $C\in\NN$ be the maximum of all $C_{j,l}$'s, then
\begin{equation*}
\begin{split}
h(\fibergraph{A}{b_{i+ni}}{\cM})&\le
2|\cM|\cdot\frac{\sum_{j\in\supp(u)}\sum_{l=0}^{\cC(\cM)}L_{j,l}(i)}{L_{\relaxfiber{A}{b}}(i)}\\
&=2|\cM|\cdot\frac{|\supp(u)|\cdot(\cC(\cM)+1)\cdot C\cdot i^{r-1}}{c_r(k)i^r
+ \sum_{s=0}^{r-1}c_s(i)i^s}.
\end{split}
\end{equation*}
Since $|\cM|$, $\cC(\cM)$,$n$, $C$, and $\supp(u)$ are constants which
are independent of $i\in\cI$ and since $c_r(k)>0$,
$(h(\fibergraph{A}{b_{i+ni}}{\cM}))_{i\in\cI}\in\mathcal{O}(\frac{1}{(n+1)i})_{i\in\cI}$.
It follows from Proposition~\ref{p:EdgeExpansion} that the sequence is
no expander and if $(b_i)_{i\in\NN}$ has a meaningful parametrization, then
Lemma~\ref{l:MeaningfulLemma} implies that the sequence cannot mix
rapidly.
\end{proof}

\begin{remark}
Basically, Theorem~\ref{t:NotRapidlyMixing} shows the existence of
$C,C'\in\NN_{\ge 1}$ such that the second largest eigenvalue modulus
$\lambda_i$ of the simple walk on $\fibergraph{A}{b_i}{\cM}$
satisfies: $\lambda_i\ge 1-\frac{C}{i}$ for all $i\in C'\cdot\NN$.
Here, the constant $C'$ is due to the many boundary effects and the
fluctuations in Ehrhart quasi-polynomials. For instance, when
$(b_i)_{i\in\NN}$ is a ray instead of a dominated sequence and when
$A$ is totally unimodular, then $\lambda_i\ge 1-\frac{C}{i}$ holds for
all $i\in\NN_{\ge C''}$ for a constant $C''\in\NN$. Also, when
$(b_i)_{i\in\NN}$ is a sequence with $\limsup_{i\in\NN}
H_A(b_i)=\infty$, then $\limsup_{i\in\NN}\lambda_i=1$ due to
Theorem~\ref{t:NotRapidlyMixing} and
Proposition~\ref{p:DistanceToFacets}.
\end{remark}

\begin{cor}\label{c:PolynomialSubsequence}
Let $A\in\ZZ^{m\times d}$ and  let $\cM\subset\ker_\ZZ(A)$ be a Markov
basis for $A$. Let $(b_i)_{i\in\NN}$ from $\cone{A}$ have a meaningful
parametrization and 
suppose there is $p\in\QQ[t]$ with $p(\NN)\subseteq\NN$ such that
$(b_{p(i)})_{i\in\NN}$ is dominated.  Then
$(\fibergraph{A}{b_i}{\cM})_{i\in\NN}$ is not rapidly mixing.
\end{cor}
\begin{proof}
Clearly, there exists $C\in\NN_{>0}$ such that
$p(C\cdot(i+1))>p(C\cdot i)$ for all $i$ sufficiently large. Let
$b'_i:=b_{p(C\cdot i)}$, then $(b'_{i})_{i\in\NN}$ is a
subsequence of $(b_i)_{i\in\NN}$ and hence it suffices to show that
$(\fibergraph{A}{b_i'}{\cM})_{i\in\NN}$ is not rapidly mixing.
Since 
$|\fiber{A}{b_i'}|=|\fiber{A}{b_{p(C\cdot i)}}|\le q(p(C\cdot i))$ for
a polynomial $q\in\QQ[t]$, $(b'_i)_{i\in\NN}$ 
has a meaningful
parametrization.
By assumption, there exists $b\in\cone{A}$ such that $b_{p(i)}-i\cdot
b\in\cone{A}$ for all $i\in\NN$ and hence 
$b'_i-i\cdot C\cdot b\in\cone{A}$ for all $i\in\NN$.
Theorem~\ref{t:NotRapidlyMixing} then implies that $(\fibergraph{A}{b'_i}{\cM})_{i\in\NN}$ is not
rapidly mixing.
\end{proof}

\begin{cor}\label{c:MixingAlongRay}
Let $A\in\ZZ^{m\times d}$, $\cM\subset\ker_\ZZ(A)$ a Markov basis for
$A$, and $b\in\NN A$ with $\dim(\relaxfiber{A}{b})>0$. Suppose that
$(b_i)_{i\in\NN}$ has $(p(k)\cdot b)_{k\in\NN}$ as subsequence for
some non-constant $p\in\QQ[t]$. Then
$(\fibergraph{A}{b_i}{\cM})_{i\in\NN}$ is not rapidly mixing.
\end{cor}
\begin{proof}
Let $C,N\in\NN$ such that $p(C\cdot k)\ge k$ for $k\ge N$.  Then
$(p(C\cdot k)-k)\cdot b\in\cone{A}$ for $k\ge N$ and hence $(p(C\cdot
k)\cdot b)_{k\in\NN_{\ge N}}$ is dominated. Clearly, $(p(C\cdot
k)\cdot b)_{k\in\NN_{\ge N}}$ is meaningfully parametrized because of
Proposition~\ref{p:MeaningfulParametrization} and hence the statement
is a consequence of Theorem~\ref{t:NotRapidlyMixing}.
\end{proof}

\begin{remark}\label{r:AddingMultiples}
Let $\cM=\{m_1,\ldots,m_r\}\subset\ker_\ZZ(A)$ be a Markov basis for
$A$. Extending $\cM$ by adding a finite number of $\ZZ$-linear
combinations $\sum_{i=1}^k\lambda_i m_i$ may improve the mixing
behaviour in one particular fiber, but since the complexity of the new
set of moves is still finite, this cannot lead to rapid mixing
asymptotically due to Theorem~\ref{t:NotRapidlyMixing}. For instance,
this implies that the Graver basis of $A$ has the same asymptotic
mixing behaviour than any other finite Markov basis for $A$.
\end{remark}

\begin{example}\label{ex:IndepModNoExpander}
Let $A_{n|m}$ be the constraint matrix of the $n\times m$-independence
model~\cite{Drton2008}. Elements in the kernel of $A_{n|m}$ can be written as $n\times
m$ contingency tables whose row and column sums are zero and the
\emph{basic moves} $\cM_{n|m}$ are a
minimal Markov basis for $A_{n|m}$. These are all moves in the orbit of
\begin{equation*}
\begin{bmatrix}
1 & -1 & 0 & \cdots & 0\\
-1 & 1 & 0 & \cdots & 0\\
0 & 0 & 0 & \\
\vdots & \vdots &  &\ddots & \vdots \\
0 & 0 & & \cdots & 0
\end{bmatrix}\in\ZZ^{n\times m}
\end{equation*}
under the group action of $S_n\times S_m$ on the rows and columns.
Using this Markov basis to explore the set of contingency tables was
suggested in \cite{Diaconis1998a}. Elements in the same fiber of
$A_{n|m}$ have the same $\|\cdot\|_1$-norm, namely
$\frac{1}{2}\|b\|_1$, since the row-space of $A_{n|m}$ contains the
vector $(1,\dots,1)\in\RR^{n\cdot m}$. Observe that the invariant
$\frac{1}{2}\|b\|_1$ is precisely the \emph{sample size} of
goodness-of-fit tests for the independence
model~\cite[Chapter~1.1]{Drton2008}.
Thus, we obtain sequences $(b_i)_{i\in\NN}$ with a
meaningful parametrization whenever the sample size grows polynomial
in $i$ by Proposition~\ref{p:MeaningfulParametrization}. Assume that
$n\ge m$ and that $b_i\ge\frac{s}{t}\cdot i\cdot
(1,\dots,1)^T$ for fixed $s,t\in\NN$, then $b_{i\cdot t\cdot n}-i\cdot
s\cdot (n,\dots,n,m,\dots,m)^T\in\cone{A_{n|m}}$ (where
$n,\dots,n$ denotes the $m$ column sums and $m,\dots,m$ denotes the
$n$ row sums) and it
follows that
$(b_{i\cdot t\cdot n})_{i\in\NN}$ is dominated since the fiber of $(n,\dots,n,m\dots,m)^T$
contains an element with full support.
Corollary~\ref{c:PolynomialSubsequence} shows that the simple fiber
walk on $(\fibergraph{A_{n|m}}{b_i}{\cM_{n|m}})_{i\in\NN}$ is
not rapidly mixing. These assumptions hold for instance when $n=m$ and
$b_i:=(i,\dots,i)\in\NN^{2n}$, even though the node-connectivity
under the basic moves $\cM_{n|n}$ is best-possible due
to~\cite[Theorem~2.9]{Potka2013}.
\end{example}

Using Markov chain comparison methods as in~\cite{Dyer2006},
Theorem~\ref{t:NotRapidlyMixing} can be used to show that related
random walks on fibers are not rapidly mixing as well.
We show that the random walk on
a fiber where we sample uniformly from the full Markov basis has
asymptotically the same mixing behaviour than the random walk which
samples from the set of \emph{applicable} moves locally.

\begin{prop}\label{p:OnlyApplicableMoves}
Let $G=(V,E)$ be a connected $d$-regular graph and let $G'$ be the
graph obtained from $G$ after removing all its loops. Let $\lambda$
and $\lambda'$ be the second largest eigenvalues of $S^G$ and $S^{G'}$
respectively, then $(1-\lambda')\le d\cdot (1-\lambda)$.
\end{prop}
\begin{proof}
Let $m$ be the number of edges and $\delta$ be the minimal degree in $G'$
respectively. If
$G'$ is bipartite, then $\lambda'=1$ and the claim
holds. Assume differently, the
stationary distribution $\pi$ of $S^{G}$ is the uniform distribution on
$V$ whereas the stationary distribution of $S^{G'}$ is
$\pi':V\to[0,1]$, $\pi'(u)=\deg_{G'}(u)\cdot(2m)^{-1}$. We use~\cite[Lemma~2.5]{Ullrich2012}. For any
$u\in V$,
\begin{equation*}
\frac{\pi'(u)}{\pi(u)}=\frac{|V|\cdot\deg_{G'}(u)}{2m}\ge\frac{|V|\cdot\delta}{2m}
\end{equation*}
and for any distinct $w,v\in V$,
\begin{equation*}
\frac{\pi'(w)\cdot S^{G'}_{w,v}}{\pi(w)\cdot
S^G_{w,v}}=\frac{|V|\cdot d}{2m}.
\end{equation*}
Since the diagonal entries of $S^{G'}$ are zero,
\cite[Lemma~2.5]{Ullrich2012} implies 
$(1-\lambda')\le d\cdot\delta^{-1}\cdot(1-\lambda)$. Since
$G$ is connected, $G'$ has no isolated nodes and hence $\delta^{-1}\le 1$.
\end{proof}

\begin{remark}\label{r:OnlyApplicableMoves}
Fix a constraint matrix $A\in\ZZ^{m\times d}$ and Markov basis $\cM$ of
$A$ and consider the random walk $S'$ on $\fibergraph{A}{b}{\cM}$
which samples for any $v\in\fiber{A}{b}$ uniformly from the set of all applicable
moves $\{m\in\pm\cM: v+m\in\NN^d\}$ to explore the fiber. This
modified random walk is precisely the simple walk on the graph
obtained from $\fibergraph{A}{b}{\cM}$ after removing all its loops.
In particular, this random walk has no rejections.  However,
Proposition~\ref{p:OnlyApplicableMoves} implies that whenever
$(\fibergraph{A}{b_i}{\cM})_{i\in\NN}$ is not rapidly mixing, this
modified random walk is not rapidly mixing as well.
\end{remark}

\section{Constructing expander graphs on fibers}\label{s:ConstructExpanders}

The message from the previous section is that the moves in a Markov
bases do not suffice to provide a good mixing behaviour
asymptotically.  A possible way out is to adapt the Markov basis
appropriately so that its complexity grows with the size of the
right-hand entries.  This
can be achieved by adding a varying number of $\ZZ$-linear
combinations of the moves in a way that the edge-expansion of the
resulting graph can be controlled. However, a growth of the set of
allowed moves comes along with an increase of the number of loops,
i.e. an increase of the rejection rate of the walk. Let
$A\in\ZZ^{m\times d}$ be a matrix,
$\cM=\{m_1,\ldots,m_k\}\subset\ker_\ZZ(A)$ be a Markov basis for $A$,
and $b\in\cone{A}$. For $l\in\NN$, let
\begin{equation*}
\cM(l)=\left\{\sum_{j=1}^k\lambda_jm_j:\lambda_1,\ldots,\lambda_k\in\ZZ, \sum_{j=1}^k|\lambda_j|\le
l\right\}
\end{equation*}
and define $\fdiam{\cM}{A}{b}:=\diam{\fibergraph{A}{b}{\cM}}$ and
$\cM^b:=\cM(\fdiam{\cM}{A}{b})$.  Using $\cM^b$ instead of $\cM$ as
a set of allowed moves, the corresponding fiber graph
$\fibergraph{A}{b}{\cM^b}$ is the complete graph on~$\fiber{A}{b}$.
We discuss in Remark~\ref{r:SamplingFromPowerMoves} how moves from
$\cM^b$ can be sampled uniformly. The transition matrix of the simple
walk on $\fibergraph{A}{b}{\cM^b}$ is
\begin{equation*}
\frac{1}{|\cM^b|}\begin{bmatrix}
1 & 1 & \dots & 1 & 1 \\
1 & \ddots & &  & 1 \\
\vdots & & & &  \vdots\\
1 & & &  \ddots & 1 \\
1 & 1 & \dots & 1 & 1 \\
\end{bmatrix}+\frac{1}{|\cM^b|}
\begin{bmatrix}
|\cM^b|-|\fiber{A}{b}| & 0 & \dots & 0 & 0 \\
0 & \ddots & &  & 0 \\
\vdots & & & &  \vdots\\
0 & & &  \ddots & 0 \\
0 & 0 & \dots & 0 & |\cM^b|-|\fiber{A}{b}| \\
\end{bmatrix}.
\end{equation*}
In particular, its second largest eigenvalue modulus is
$1-\frac{|\fiber{A}{b}|}{|\cM^b|}$ and hence the next proposition is
immediate.

\begin{prop}\label{p:FiberExpanders}
Let $A\in\ZZ^{m\times d}$, $\cM\subset\ker_\ZZ(A)$ be a Markov
basis for $A$ and $(b_i)_{i\in\NN}$ a sequence in $\cone{A}$. Suppose
there exists $r\in\NN$
such that $(|\fiber{A}{b_i}|)_{i\in\NN}\in\Omega(i^r)_{i\in\NN}$ and
$(|\cM^{b_i}|)_{i\in\NN}\in\mathcal{O}(i^r)_{i\in\NN}$,
then $(\fibergraph{A}{b_i}{\cM^{b_i}})_{i\in\NN}$ is an
expander.
\end{prop}

To make use of Proposition~\ref{p:FiberExpanders}, the growths of the
fibers and the adapted Markov bases have to be compared. Again,
Ehrhart's theory applies to compute the growth of certain fiber
sequences. The asymptotic growth of~$\cM^{b_i}$ depends on the growth
of the diameter of $\fibergraph{A}{b_i}{\cM}$. Hence, we first want to
understand how the number of elements in $\cM(l)$ grows as a function
of $l\in\NN$.

\begin{lemma}\label{l:GrowthPowerSet}
Let $\cM=\{m_1,\ldots,m_k\}\subset\ZZ^d$, then 
$(|\cM(l)|)_{l\in\NN}\in\mathcal{O}(l^{\rank(\cM)})_{l\in\NN}$.
\end{lemma}
\begin{proof}
We identify the finite set $\cM$ with the integer matrix
$(m_1,\ldots,m_k)\in\ZZ^{d\times k}$. Denote the $k$-dimensional
cross-polytope by $\cP:=\{x\in\QQ^k: \|x\|_1\le
1\}$ and let $\cP':=\{\cM\cdot x:
x\in\cP\}$ be its image in $\QQ^d$ under $\cM$. 
With this, we can write $\cM(l)=\{\cM\cdot x: x\in
(l\cdot\cP)\cap\ZZ^k\}$ and hence
$\cM(l)\subseteq(l\cdot\cP')\cap\ZZ^d$. Since $\cP'$ is a polytope, Ehrhart's theorem
\cite[Theorem~3.23]{Beck2007}
gives $|(l\cdot\cP')\cap\ZZ^d|\le C\cdot l^{\dim(\cP')}$ for some
$C\in\QQ_{>0}$ and since $\dim(\cP')=\rank(\cM)$, the claim
follows. 
\end{proof}

\begin{cor}\label{c:LinearDiameter}
Let $A\in\ZZ^{m\times d}$ and let $\cM\subset\ker_\ZZ(A)$ be a Markov
basis for $A$. Let $(b_i)_{i\in\NN}$ be a sequence in $\cone{A}$ such
that $(|\fiber{A}{b_i}|)_{i\in\NN}\in\Omega(i^{d-\rank(A)})$
and $(\fdiam{\cM}{A}{b_i})_{i\in\NN}\in\mathcal{O}(i)_{i\in\NN}$. Then
$(\fibergraph{A}{b_i}{\cM^{b_i}})_{i\in\NN}$ is an
expander. 
\end{cor}
\begin{proof}
Let $r:=\dim(\ker_\ZZ(A))$. It suffices to show that $|\cM^{b_i}|\le
C\cdot i^r$ for a constant $C\in\QQ_{\ge 0}$ since the statement
follows then from Proposition~\ref{p:FiberExpanders}. Since $\cM$ is a
Markov basis for $A$, $\rank(\cM)=r$ and thus
Lemma~\ref{l:GrowthPowerSet} implies that $|\cM(l)|\le C_1\cdot l^r$
for a constant $C_1\in\QQ_{\ge 0}$. The assumption implies that there
exists $C_2\in\QQ_{\ge 0}$ such that $\fdiam{\cM}{A}{b_i}\le C_2\cdot
i$ for all $i\in\NN$. Then,
$|\cM^{b_i}|=|\cM(\fdiam{\cM}{A}{b_i})|\le|\cM(C_2\cdot i)|\le
C_1\cdot C_2^r\cdot i^r$.
\end{proof}

Expanders are not per se fast, and Corollary~\ref{c:LinearDiameter} is
an asymptotic statement. That means, for a given matrix
$A\in\ZZ^{m\times d}$, a given Markov basis $\cM\subset\ker_\ZZ(A)$,
and a right-hand side $b\in\cone{A}$, we know by
Theorem~\ref{t:NotRapidlyMixing} that the second largest eigenvalue
modulus of the simple walk that uses $\cM$ can be arbitrarily close to
$1$. On the other hand, since $(\fdiam{\cM}{A}{i\cdot
b})_{i\in\NN}\in\mathcal{O}(i)_{i\in\NN}$
by~\cite{windisch2016-heatbath}, the second largest eigenvalue modulus
of the simple walk that uses the adapted Markov basis $\cM^{i\cdot b}$
can be bounded away from $1$ strictly.  Thus, there exists a threshold
$i_0\in\NN$  such that the adapted Markov basis is faster than the
conventional Markov basis on $\fiber{A}{i\cdot b}$ for $i\ge i_0$.
The exact value of $i_0$ depends on the hidden constants in the
asymptotic formulations of Corollary~\ref{c:LinearDiameter} and can be
quite small, as in Figure~\ref{f:AdaptedVsConventional}, but also very
large so that the advantages of the adapted Markov bases may pay off
only for large right-hand sides.

\begin{remark}\label{r:SamplingFromPowerMoves}
Running the simple walk on $\fibergraph{A}{b}{\cM(l)}$ for some
$l\in\NN$ requires to sample from $\cM(l)$ uniformly and hence a good
understanding of this set is necessary.  Basically, we shift the
problem of sampling from $\fiber{A}{b}$ for all $b\in\cone{A}$ where
$\fibergraph{A}{b}{\cM}$ has diameter $l$ to the problem of sampling
from $\cM(l)$, which can be seen as some kind of \emph{rejection
sampling} from a larger set
$u+\cM(l)\supseteq\fiber{A}{b}$.
For large fibers, one applicable
move $m\in\cM(l)$ suffices to obtain a sample $u+m\in\fiber{A}{b}$
that is very close to uniform. 
Write $\cM=\{m_1,\dots,m_k\}$ and $r:=\rank(\cM)$. When $r=k$, then an
element $\lambda$ picked uniformly from $\{u\in\ZZ^k: \|u\|_1\le l\}$ gives
rise to an element $\cM\cdot\lambda$ that is uniformly generated from $\cM(l)$. This is
not the case when $r>k$. One approach to sample from $\cM(l)$
uniformly in this case is to first compute a lattice basis
$\cB:=\{b_1,\dots,b_r\}\subset\ZZ^d$ of $\cM\cdot\ZZ^k$
in order to get rid of relations among the moves from~$\cM$. Then, we compute for
every $i\in[k]$
coefficients $\lambda_1^i,\dots,\lambda_r^i$ such that
$m_i=\sum_{j=1}^r\lambda_j^ib_j$. For
$C:=\sum_{j=1}^r\max_{i\in[k]}|\lambda_j^i|$, we have
$\cM(l)\subseteq\cB(C\cdot l)$. Thus, after sampling coefficients $\lambda$
from $\{u\in\ZZ^r: \|u\|_1\le C\cdot l\}$ uniformly, we obtain a move
$\cB\cdot\lambda$ that is sampled uniformly from a superset of $\cM(l)$.
Since $|\cB(C\cdot l)|$ grows as
$\mathcal{O}(l^r)_{l\in\NN}$,
Proposition~\ref{p:FiberExpanders} remains valid. 
Sampling from the cross-polytope
$\{u\in\ZZ^r: \|u\|_1\le C\cdot l\}$ can be done with the heat-bath
method as studied in~\cite{windisch2016-heatbath}, which is fast for
$l\to\infty$.
\end{remark}

\begin{example}\label{ex:IndepModExpander}
The constraint matrix $A_{n|n}$ of the independence model
(Example~\ref{ex:IndepModNoExpander}) is totally unimodular and hence
$\dim(\relaxfiber{A_{n|n}}{\mathbf{1}_{n}})=\dim(\fiber{A_{n|n}}{\mathbf{1}_{n}})$
where $\mathbf{1}_n\in\ZZ^{n+n}$ is the vector with all entries equal
to~$1$. It was shown in~\cite[Proposition~2.10]{Potka2013} that the
diameter of $\fibergraph{A_{n|n}}{\cM_{n|n}}{i\cdot\mathbf{1}_n}$ is
$(n-1)i$. In particular, for fixed $n\in\NN$, the diameter grows
linearly in $i$ and hence Corollary~\ref{c:LinearDiameter} yields that
the sequence
$(\fibergraph{A_{n|n}}{i\cdot\mathbf{1}_{n}}{\cM_{n|n}^{i\cdot\mathbf{1}_{n}}})_{i\in\NN}$
is an expander.
\end{example}

\begin{example}\label{ex:ExpanderOnRunningExample}
Assume $d>2$ and consider $A_d$ and $\cM_d$ from
Example~\ref{ex:NoExpander}. It is not hard to see that the
graph-distance between any two nodes $u,v\in\fiber{A_d}{i}$ is at most
$\|u-v\|_1$. Since the maximal
$\|\cdot\|_1$-distance of two elements in $\fiber{A_d}{i}$ is $2i$,
the diameter of $\fibergraph{A_d}{i}{\cM_d}$ is $2i$. Hence,
$(\fibergraph{A_d}{i}{\cM_d^i})_{i\in\NN}$ is an expander.
\end{example}

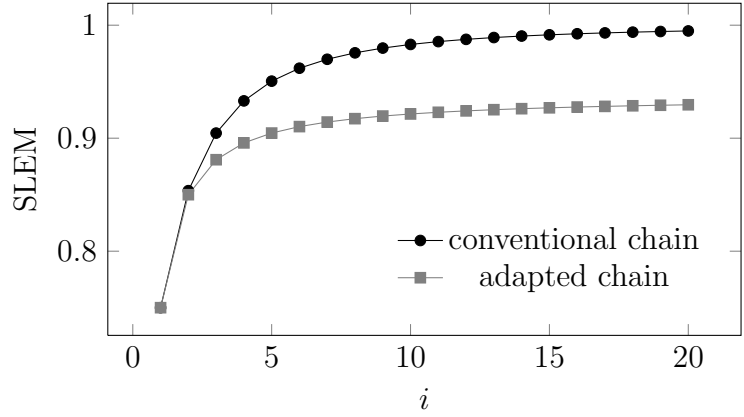
\begin{figure}[htbp]
\begin{tikzpicture}
	\begin{axis}[
		xlabel={$i$},
		ylabel={SLEM},
		legend style={at={(0.7,0.1)},anchor=south,draw=none},
		legend entries={conventional chain, adapted chain}]
		\addplot+[black,mark options={fill=black}] table {\fPartition};
		\addplot+[gray,mark options={fill=gray}]table {\fPartitionAdapted};
	\end{axis}
\end{tikzpicture}
\caption{The SLEM of the simple walk on $\fiber{A_3}{i}$ using
moves from the conventional Markov basis $\cM_3$ and the adapted moves
$\cM_3(2i)$.}\label{f:AdaptedVsConventional}
\end{figure}

\begin{example}\label{ex:ExpanderOnHemmecke}
For $k\in\NN$, let $I_k$ be the identity
matrix in $\ZZ^{k\times k}$, $\mathbf{1}_k$ be the $k$-dimensional
vector with all entries equal to $1$ and define the matrix 
\begin{equation}\label{equ:HemmeckeMat}
H_k:=\begin{bmatrix}
I_k & I_k & \mathbf{0} & \mathbf{0} & -\mathbf{1}_k & \mathbf{0} \\
\mathbf{0} & \mathbf{0} & I_k & I_k &  \mathbf{0}  & -\mathbf{1}_k\\
\mathbf{0} & \mathbf{0} & \mathbf{0} & \mathbf{0} & 1 & 1
\end{bmatrix}\in\ZZ^{(2k+1)\times(4k+2)}.
\end{equation}
It was shown in~\cite[Theorem~2]{Hemmecke2014} that the reduced
lexicographic Gröbner basis $\cG_k$ of $H_k$ is
$\{e_i-e_{k+i}:i\in\{1,\ldots,k,2k+1,\ldots,3k\}\}$ together with the
move
\begin{equation*}
\begin{pmatrix}0,\ldots,0,1,\ldots,1,0,\ldots,0,-1,\ldots,-1,1,-1\end{pmatrix}^T.
\end{equation*}
With \cite[Section~4]{Hemmecke2014}, it is easy to
show that for any $k\in\NN$, the diameter of $\fibergraph{H_k}{i\cdot
e_{2k+1}}{\cG_k}$ is $(2k+1)i$. Thus,
$(\fibergraph{H_k}{ie_{2k+1}}{\cG_k((2k+1)i)})_{i\in\NN}$ is
is an expander.
\end{example}

\section{Scaling the dimension}

Markov bases of constraint matrices coming from statistical problems
are often parametrized and they can be stated explicitly for any
parameter. For instance, the basic moves $\cM_{n|n}$ of the
independence model\index{log-linear model!independence model}
(Example~\ref{ex:IndepModNoExpander}) form a Markov basis for
$A_{n|n}$ for every $n\in\NN$.  Thus, varying the parameter $n$
provides fiber graphs where the set of moves is adapted canonically.

\begin{remark}
Let $b_n:=(1,\dots,1)\in\NN^{2n}$, then
the elements of $\fiber{A_{n|n}}{b_n}$ can be
identified with the elements of the symmetric group $S_n$ on $[n]$. 
Finding a
set of generators such that the corresponding \emph{Cayley graph}\index{graph!Cayley
graph} on $S_n$ is an expander is an active research field in group theory,
see for instance~\cite{kassabov2007}. 
In~\cite{diaconis1981}, it was shown that the simple walk on the Cayley
graph of $S_n$ that uses the transpositions mixes rapidly in
$\frac{1}{2}n\log n$ many steps.
Inspired by shuffling a deck of $n$ cards, a random walk on~$S_n$ that
uses \emph{riffle shuffles} was studied
in~\cite{bayer1992} and shown to be rapidly mixing as well.
\end{remark}

Parametric descriptions of Markov bases can be arbitrarily complicated
in general, since by the Universality theorem~\cite{deloera2006}, any integer vector appears as a subvector
of a Markov basis element of the three-way no interaction model, when
the parameters are large enough. Different than in fixed
dimension, where the Markov basis is fixed, the size of the Markov
basis is important in the convergence analysis when the dimension
varies because the local sampling process of a move can be
computationally challenging as the
Markov basis becomes larger. The trade-off
between an easily accessible set of moves and a corresponding random
walk that has good mixing properties shows the
realms of fiber walks in practice. The next proposition illustrates
this for
$H_k$ from Example~\ref{ex:ExpanderOnHemmecke}, where the overwhelming
number of moves in its parametric Graver basis $\graver{k}$ slows the chain down
for $k\to\infty$, despite the fact that the
edge-connectivity of these fibers
is best-possible~\cite[Theorem~4]{Hemmecke2014}. A description
of $\graver{k}$ is in~\cite[Theorem~2]{Hemmecke2014}.

\begin{prop}\label{p:ScalingDimensionOfHemmeckeWindisch}
The sequence $(\fibergraph{H_k}{e_{2k+1}}{\graver{k}})_{k\in\NN}$ is
not rapidly mixing.
\end{prop}
\begin{proof}
According to \cite[Section 4]{Hemmecke2014},
$\fibergraph{H_k}{e_{2k+1}}{\graver{k}}$ is isomorphic to the graph on
the nodes
$\{0,1\}^{k+1}$ in which two nodes
$(i_1,\ldots,i_{k+1})$ and $(j_1,\ldots,j_{k+1})$ are adjacent if either
$i_{k+1}=j_{k+1}$ and $\|i-j\|_{\infty}=1$, or if $i_{i+1}\neq
j_{k+1}$. For any $k\in\NN_{>0}$, let $S_k:=\{(0,i,0):
i\in\{0,1\}^{k-1}\}\cup\{(0,i,1):i\in\{0,1\}^{k-1}\}$, then
$|S_k|=\frac{1}{2}|\fiber{A_k}{e_{2k+1}}|$. Counting the edges leaving
$S_k$, for any $(0,i,0)\in S_k$ there are $k$ many with endpoints in
$\{(1,i,0):i\in\{0,1\}^{k-1}\}$ and $2^{k-1}$ with endpoints in
$\{(1,i,1):i\in\{0,1\}^{k-1}\}$. The same is true for any $(0,i,1)\in S_k$.
Hence, there are $(k+2^{k-1})\cdot 2\cdot 2^{k-1}$ edges leaving~$S_k$. The
edge-expansion of $\fibergraph{A_k}{e_{2k+1}}{\graver{k}}$  is thus
bounded from above by $k+2^{k-1}$. Since
$|\graver{k}|=2\cdot(4^k+4k)$ and
$\log|\fiber{H_k}{{e_{2k+1}}}|=k+1$, the claim follows.
\end{proof}

\bibliographystyle{amsplain}
\bibliography{fiberWalks}

\end{document}